\documentclass[12pt]{amsart}

\usepackage{amsmath,amsfonts,amssymb}

\newtheorem{theorem}{Theorem}[section]
\newtheorem{corollary}[theorem]{Corollary}
\newtheorem{lemma}{Lemma}
\newtheorem{proposition}[theorem]{Proposition}

\numberwithin{equation}{section}

\makeatletter \@namedef{subjclassname@2010}{%
  \textup{2010} Mathematics Subject Classification}
\makeatother

\begin{document}

\baselineskip=17pt

\title[A generalization of a theorem of Erd\H{o}s-R\'enyi]{A generalization of a theorem of Erd\H{o}s-R\'enyi \\
to $m$-fold sums and differences}
\author[K.E. Hare]{Kathryn Hare}
\address{Dept. of Pure Mathematics\\
University of Waterloo\\ Waterloo, Ont.\\ Canada N2L 3G1} \email{kehare@uwaterloo.ca}
\thanks{This research is partially supported by NSERC}

\author[S. Yamagishi]{Shuntaro Yamagishi}
\address{Dept. of Pure Mathematics\\
University of Waterloo\\ Waterloo, Ont.\\ Canada N2L 3G1} \email{syamagis@uwaterloo.ca}

\begin{abstract}
Let $m\geq 2$ be a positive integer. Given a set $E(\omega )\subseteq \mathbb{N}$ we define $r_{N}^{(m)}(\omega )$ to be the number of ways to
represent $N\in \mathbb{Z}$ as any combination of sums\textit{\ and} differences of $m$ distinct elements of $E(\omega )$. In this paper, we prove the
existence of a ``thick'' set $E(\omega )$ and a positive constant $K$ such that $r_{N}^{(m)}(\omega )<K$ for all $N\in \mathbb{Z}$. This is a
generalization of a known theorem by Erd\H{o}s and R\'enyi. We also apply our results to harmonic analysis, where we prove the existence of certain
thin sets.
\end{abstract}



\maketitle

\section{Introduction}

Given a set $S \subseteq \mathbb{N}$, we define $R_S^{m}(n)$ 
to be the number of representations of the form $n = s_1 + ... + s_m \ (s_i \in S)$ and $s_1 \leq ... \leq s_m$.
We say that the
set $S$ is of type $B_m(g)$ if
\begin{equation*}
R_S^{m}(n) \leq g
\end{equation*}
for all $n$. In \cite{Vu}, Vu gives a brief history of the topic, which we paraphrase here. In 1932, Sidon, in connection with his work in Fourier
analysis, investigated power series of type $\sum_{i=1}^{\infty} z^{a_i}$, when $(\sum_{i=1}^{\infty} z^{a_i})^m$ has bounded coefficients \cite{Sid}.
This leads to the study of sets of type $B_m(g)$. One classical question in this area is the following \cite{PS}.

``Let $S$ be a set of type $B_{m}(g)$. How fast can $S(n)$ grow, where $S(n)$ is the number of elements of $S$ not exceeding $n$ ?''

In \cite{ER}, Erd\H{o}s and R\'enyi gave an answer to this question for the case $m=2$. This result was discussed in great detail in the monograph of
Halberstam and Roth \cite{HR}.

\begin{theorem}
[Erd\H{o}s-R\'enyi] \label{thm ER} For any $\varepsilon > 0$, there exists $g = g(\varepsilon)$ and a set $S \subseteq \mathbb{N}$ of type $B_2(g)$
such that
\begin{equation*}
S(n) > n^{\frac12 - \varepsilon}
\end{equation*}
for sufficiently large $n$.
\end{theorem}

The result is best possible up to the $\varepsilon $ term in the exponent. Erd\H{o}s-R$\acute{\text{e}}$nyi
used a probabilistic argument, and their proof was
presented in \cite{HR}
in a more rigorous and carefully written form.
This theorem can be generalized from $2$-fold sums to the following theorem
for arbitrary $m$-fold sums, as was noted in \cite{Ha} and \cite{HR} (without proof), and also by Vu who observed that it can be deduced as a
consequence of a more general result proven in \cite{Vu}.

\begin{theorem}
\label{thm m fold} For any positive integer $m\geq 2$ and any $\varepsilon
>0 $, there exists $g=g(\varepsilon ,m)$ and a set $S\subseteq \mathbb{N}$
of type $B_{m}(g)$ such that
\begin{equation*}
S(n)>n^{\frac{1}{m}-\varepsilon }
\end{equation*}
for sufficiently large $n$.
\end{theorem}

Given a set $E(\omega )\subseteq \mathbb{N}$ we define $r_{N}^{(m)}(\omega )$ to be the number of ways to represent $N\in \mathbb{Z}$ as any
combination of sums \textit{and} differences of $m$ distinct elements of $E(\omega )$. In this paper, we prove the existence of a ``thick'' set
$E(\omega )$ and a positive constant $K$ such that $r_{N}^{(m)}(\omega )<K$ for all $N\in \mathbb{Z}$. Hence, our theorem is a (partial) generalization
of Theorem \ref {thm m fold}. The caveat here is that with $r_{N}^{(m)}(\omega )$ we do not allow repeated elements in the representation, for
otherwise every infinite set will admit integers $N$ with $r_{N}^{(m)}(\omega )=\infty $. (Take $N=0$ when $m$ is even, for instance.)

Our main result is the following:

\begin{theorem}
\label{thm main intger} For any positive integer $m\geq 2$ and $\varepsilon
>0$, there exists $K=K(\varepsilon ,m)$ and a set $E(\omega )\subseteq
\mathbb{N}$ such that
\begin{equation*}
r_{N}^{(m)}(\omega )<K
\end{equation*}
for all $N\in \mathbb{Z}$, and
\begin{equation*}
card\left\{ E(\omega )\bigcap \{1,2,\dots,n\}\right\} \gg n^{\frac{1}{m} -\varepsilon }
\end{equation*}
for sufficiently large $n$.
\end{theorem}

We will also prove analogous results for $\bigoplus_{0}^{\infty }\mathbb{Z} (q)$, $\bigoplus_{0}^{\infty }\mathbb{Z}(q_{n})$ for $\{q_{n}\}$ increasing
integers, and $\mathbb{Z}(q^{\infty })$ where $q$ is prime. Here $\mathbb{Z}(m)$ denotes the
cyclic group of order $m \in \mathbb{N}$ and $\mathbb{Z}(q^{\infty })$ is the group of
all $q^n$-th roots of unity. The notation  $\bigoplus_{0}^{\infty }$ means the countable direct sum.

In section \ref{HA}, we give an application of our results to harmonic analysis.
A subset $E$ of a discrete abelian group with dual group $X$ is called a
$\Lambda(q)$ set (for some $q > 2$ ) if whenever $f \in L^2(X)$ and
the Fourier transform of $f$ is non-zero only on $E$, then $f \in L^q(X)$.
A completely bounded $\Lambda(q)$ set is defined in a similar spirit, but is more complicated
and we refer the reader to section \ref{HA} for the definition.
We prove that for any integer $m\geq 2$ and $\varepsilon >0$, every
infinite discrete abelian group contains a set that is completely bounded $\Lambda (2m)$, but not $\Lambda (2m+\varepsilon )$.

Notation. We use $\ll$ and $\gg$ to denote Vinogradov's well-known notation.

\section{Preliminaries}

Let $G$ be any one of $\mathbb{Z}$, $\bigoplus_{0}^{\infty }\mathbb{Z}(q)$ where $q\in \mathbb{N}$, $\mathbb{Z}(q^{\infty })$ where $q$ is prime, or
$\bigoplus_{0}^{\infty }\mathbb{Z}(q_{n})$ where $\{q_{n}\}$ are strictly increasing, odd integers. (The case when not all $q_{n}$ are odd requires a
notational adaptation that we will leave for the reader.) We define $ G^{\prime }$ to be $\mathbb{N}$ if $G=\mathbb{Z}$ and $G^{\prime }=G\backslash
\{0\}$ otherwise.

If $\psi \in \bigoplus_{0}^{\infty }\mathbb{Z}(q)$, then $\psi =(\psi _{i})_{i=0}^{\infty }$, where each $\psi _{i}\in \mathbb{Z}(q)$ and all but
finitely many $\psi _{i}=0$. Given $\psi =(\psi _{i})_{i=0}^{\infty }$, we call the degree of $\psi $ (or $\deg \psi $ for short) the maximum $i$ such
that $\psi _{i}\not=0$. We also let $\deg 0=-\infty $. By choosing the representative $0\leq \psi _{i}<q$ for each $i$, we can identify $\psi \not=0 $
with the natural number $\psi _{0}+\psi _{1}q+\cdots+\psi _{d}q^{d}$, where $d=\deg \psi .$ This gives a one to one correspondence between $
\bigoplus_{0}^{\infty }\mathbb{Z}(q)$ and $\mathbb{N}\cup \{0\}$. Notice there are $q^{d+1}-q^{d}$ elements of degree $d$ for each $d\geq 0$.

If $\psi \in \mathbb{Z}(q^{\infty })$ and $\psi \not=0$, then $\psi $ is the argument of a primitive $q^{M}$-th root of unity for a unique choice of
$M$, in other words $\psi =j/q^{M}$ where $j\in \{1,2,\dots,q^{M}-1\}$ and $q\nmid j $. We let $\deg \psi =M-1$ and $\deg 0=-\infty $. Again, for each
$d\geq 0$, there are $q^{d+1}-q^{d}$ elements of degree $d$. We can identify $\mathbb{ Z}(q^{\infty })$ with $\mathbb{N}\cup \{0\}$ by assigning $0$ to
$0$, and elements of degree $d$ to $\{q^{d},q^{d}+1,\dots,q^{d+1}-1\}$ for each $d\geq 0$ in the natural way.

In the case that $G=\bigoplus_{0}^{\infty }\mathbb{Z}(q_{n})$ where $q_{n}$ are strictly increasing, odd integers we choose representatives from $
\{-(q_{n}-1)/2,\dots,(q_{n}-1)/2\}$ for each $\mathbb{Z}(q_{n})$. We define the degree of $\psi =(\psi _{i})_{i=0}^{\infty }$ in the same manner as for
$ \bigoplus_{0}^{\infty }\mathbb{Z}(q)$. We will identify the $2q_{0}\cdots q_{d-1}$ characters $\psi =(\psi _{i})$ which have degree $d$ and $d$-th
coordinate $\psi _{d}=\pm r,r\in \mathbb{N}$, with the integers in the interval $[(2r-1)q_{0}\cdots q_{d-1},(2r+1)q_{0}\cdots q_{d-1})$ (where, if
$d=0$, we understand $q_{0}\cdots q_{d-1}$ $=1$). Hence, the characters of degree $d$ are assigned to integers in $\{q_{0}\cdots
q_{d-1},\dots,q_{0}\cdots q_{d}-1\}.$

Thus, for any of the four choices of $G$ above, we have $G^{\prime }=\{\chi _{n}\}_{n=1}^{\infty }$ where $\chi _{n}$ is the non-zero element of $G$
uniquely associated with the integer $n$.

Given real numbers $\alpha _{n}$ with $0<\alpha _{n}<1$, we let $Y_{n}$, $n=1,2,\dots,$ be independent Bernoulli random variables defined on a
probability space $(\Omega ,M,P),$ with $P(Y_{n}=1)=\alpha _{n}$ and $P(Y_{n}=0)=1-\alpha _{n}$. For each of the groups $G$ we define random subsets
\begin{equation*}
E(\omega )=E(\omega ,G)=\{\chi _{n}\in G^{\prime }:Y_{n}(\omega )=1\}.
\end{equation*}
Throughout the paper we will be specifying a positive number $s$ and putting
\begin{equation}
\alpha _{n}=n^{-s}\text{ when }G=\mathbb{Z}\text{, }\bigoplus_{0}^{\infty } \mathbb{Z}(q)\text{ or }\mathbb{Z}(q^{\infty })  \label{prob1}
\end{equation}
or
\begin{equation}
\begin{split}
\alpha _{n} & =
\begin{cases}
n^{-s} & \mbox{if  }q_{0}\cdots q_{d-1}<n\leq (2\left\lfloor \frac{q_{d}}{8m}\right\rfloor +1)q_{0}\cdots q_{d-1}\text{ and } q_{d}>8m \\ 0 &
\text{else}
\end{cases} \\
& \qquad \text{when }G=\bigoplus_{0}^{\infty }\mathbb{Z}(q_{n}).  \label{prob2}
\end{split}
\end{equation}
Note we have $\alpha _{n}\leq n^{-s}$ for all $n$ in this case, as well.

Let $m\geq 2$ be a positive integer. For $t\in \{0,1,\dots,m\}$, we define
\begin{align*}
r_{N,t}^{(m)}(\omega ) :=card & \left\{ (a_{1},\dots,a_{m}):\chi _{a_{i}}\in E(\omega ),\sum_{i=1}^{t}\chi _{a_{i}}-\sum_{i=t+1}^{m}\chi _{a_{i}}= \chi
_{N}, \right.\\ & \qquad a_{1}<\cdots <a_{t},a_{t+1}<\cdots <a_{m},a_{i}\not=a_{j}\text{ for }i\not=j\Bigg\} .
\end{align*}
For clarification, we note that when $t=0$ and $t=m$, we mean to consider the expressions $-\sum_{i=1}^{m}\chi _{a_{i}}=\chi _{N}$ and
$\sum_{i=1}^{m}\chi _{a_{i}}=\chi _{N}$, respectively, with $a_{1}<\cdots<a_{m}$. We also define
\begin{equation}
r_{N}^{(m)}(\omega ):=\sum_{t=0}^{m}r_{N,t}^{(m)}(\omega ).  \label{def r}
\end{equation}
Lastly, we recall a fact about elementary symmetric functions that will be useful later.

\begin{lemma}
\label{lemma 4} Let $\{y_{k}\}_{k>0}$ be a sequence of non-negative numbers. For each $d\in \mathbb{N}$, we write
\begin{equation*}
\sigma _{d}=\sum_{k_{1}<\cdots<k_{d}}y_{k_{1}}y_{k_{2}}\cdots y_{k_{d}},
\end{equation*}
in other words, $\sigma _{d}$ is the $d$-th elementary symmetric function of the $y_{k}$. Then,
\begin{equation*}
\sigma _{d}\leq \frac{\sigma _{1}^{d}}{d!}.
\end{equation*}
\end{lemma}

\begin{proof}
See \cite[p 147, Lem 13]{HR}.
\end{proof}

\section{$m$-fold sums and differences}

Our main result, Theorem \ref{thm main intger}, follows fairly easily from Theorem \ref{theorem 6}. This will be proven by induction, with the base
case taken care of in Cor \ref{corollary 5} (following Prop \ref{prop 4}). Unless we specify otherwise in the statement of the result, in this section
$G$ can be considered to be any one of $\mathbb{Z}$, $\bigoplus_{0}^{\infty }\mathbb{ Z}(q)$, $\mathbb{Z}(q^{\infty })$ or $\bigoplus_{0}^{\infty
}\mathbb{Z} (q_{n})$.

We begin with some observations utilized in the proof of the next lemma: Suppose $\chi _{n}=\sum_{i=1}^{m}\varepsilon _{i}\chi _{n_{i}}$, where
$\varepsilon _{i}=\pm 1$, $n_{i}$ are distinct and all $\alpha _{n_{i}}\neq 0$.

When $G=\mathbb{Z}$, $\max_{1\leq i\leq m}\left\vert n_{i}\right\vert \geq \left\vert n\right\vert /m$.

When $G=\bigoplus_{0}^{\infty }\mathbb{Z}(q)$, $\mathbb{Z}(q^{\infty })$ (or $\bigoplus_{0}^{\infty }\mathbb{Z}(q_{n})$) and $\deg \chi _{n}=d$, then
$q^{d}\leq n\leq q^{d+1}$ (or $q_{0}\cdots q_{d-1}\leq n\leq q_{0}\cdots q_{d}$). Since $\deg (\chi _{a}\pm \chi _{b})\leq \max \{\deg \chi _{a},\deg
\chi _{b}\}$, it follows that $\max_{1\leq i\leq m}\deg \chi _{n_{i}}\geq \deg \chi _{n}$.

When $G=\bigoplus_{0}^{\infty }\mathbb{Z}(q_{n})$ and $\chi _{n}=(\chi _{n,j})_{j=1}^{\infty }$ with degree $d\geq 0$, then
\begin{equation*}
n\leq (2\left\vert \chi _{n,d}\right\vert +1)q_{0}\cdots q_{d-1}.
\end{equation*}
If $\max_{1\leq i\leq m}\deg \chi _{n_{i}}=\deg \chi _{n}$, (which can occur only if $q_{d}>8m$ as we are assuming all $\alpha _{n_{i}}\neq 0$) then
the modulus of the $d$-th coordinate of $\chi _{n_{i}}$ is at least $\left\lfloor \left\vert \chi _{n,d}\right\vert /m\right\rfloor $ for some $\chi
_{n_{i}}$ of maximal degree. This is because addition in the $d$-th coordinate on the terms where $\alpha _{n_{i}}\neq 0$ is the same as in
$\mathbb{Z}$ (recall (\ref{prob2})). Thus, (whether $\max_{1\leq i\leq m}\deg \chi _{n_{i}}>\deg \chi _{n}$ or $\max_{1\leq i\leq m}\deg \chi
_{n_{i}}=\deg \chi _{n}$) there must be some $i$ such that
\begin{equation*}
n_{i}\geq (2\left\lfloor \left\vert \chi _{n,d}\right\vert /m\right\rfloor -1)q_{0}\cdots q_{d-1}\geq \frac{\left\vert \chi _{n,d}\right\vert
}{m}q_{0}\cdots q_{d-1}.
\end{equation*}

\begin{lemma}
\label{lemma 3} Given $m\geq 2$, we let $s=\frac{m-1}{m}+\theta $ where $0<\theta <\frac{1}{m}$, and let $\alpha _{n}$ be as in (\ref{prob1}) and
(\ref{prob2}). Fix $\varepsilon _{1},\varepsilon _{2},\dots,\varepsilon _{m}\in \{\pm 1\}$. We have
\begin{equation}
\sum_{\overset{(n_{1},\dots,n_{m})\in \mathbb{N}^{m}}{\sum_{i=1}^{m} \varepsilon _{i}\chi _{n_{i}}=\chi _{n}}}\alpha _{n_{1}}\cdots\alpha _{n_{m}}\leq
\
\begin{cases}
C_{m}, & \mbox{if }n=0, \\ C_{m}\frac{1}{\left\vert n\right\vert ^{m\theta }}, & \mbox{if }n\not=0,
\end{cases}
\label{eqn lemma 3-2}
\end{equation}
where $C_{m}$ is a positive constant dependent only on $G$, $m$ and $s$.
\end{lemma}

\begin{proof}
We only give the proof for $G=\bigoplus_{0}^{\infty }\mathbb{Z}(q)$ and $\bigoplus_{0}^{\infty }\mathbb{Z}(q_{n})$ as the other cases can be obtained
by similar calculations. In both cases we use the fact that $\sum_{n\leq A}n^{-s}\ll A^{1-s}$.

Case $G=\bigoplus_{0}^{\infty }\mathbb{Z}(q)$: If $n\not=0$, let $t_{0}\geq 0 $ be such that $n\in \lbrack q^{t_{0}},q^{t_{0}+1})$. If $n=0$, we let
$t_{0}=0$. For either case, we have $\max_{1\leq i\leq m}\deg \chi _{n_{i}}\geq t_{0}$ for any choice of $(n_{1},\dots,n_{m})$ in the summand of
(\ref{eqn lemma 3-2}). Therefore, we can simplify and bound the sum (\ref {eqn lemma 3-2}) as follows,
\begin{align}
\sum_{\overset{(n_{1},\dots,n_{m})\in \mathbb{N}^{m}}{\sum_{i=1}^{m} \varepsilon _{i}\chi _{n_{i}}=\chi _{n}}}\alpha _{n_{1}}\cdots\alpha _{n_{m}}
&\leq \sum_{t=t_{0}}^{\infty }\sum_{\overset{\sum_{i=1}^{m}\varepsilon _{i}\chi _{n_{i}}=\chi _{n}}{\max {n_{i}}\in \lbrack q^{t},q^{t+1})}}\alpha
_{n_{1}}\cdots\alpha _{n_{m}}  \label{calc lemma 3-2} \\ &\leq \sum_{t=t_{0}}^{\infty }\frac{1}{q^{ts}}\sum_{\overset{n_{i}<q^{t+1}}{ 1\leq i<m}}\alpha
_{n_{1}}\cdots\alpha _{n_{m-1}}  \notag \\ &\leq \sum_{t=t_{0}}^{\infty }\frac{1}{q^{ts}}\left( \sum_{n<q^{t+1}}\alpha _{n}\right) ^{m-1}.  \notag
\end{align}
As $\alpha _{n}\leq n^{-s}$ we deduce that
\begin{equation*}
\sum_{\overset{(n_{1},\dots,n_{m})\in \mathbb{N}^{m}}{\sum_{i=1}^{m} \varepsilon _{i}\chi _{n_{i}}=\chi _{n}}}\alpha _{n_{1}}\cdots\alpha _{n_{m}}\leq
C_{m}q^{-t_{0}m\theta },
\end{equation*}
which is equivalent to the inequality we desired to show.

Case $G=\bigoplus_{0}^{\infty }\mathbb{Z}(q_{n})$:  Since $\alpha _{n_{1}}\cdots\alpha _{n_{m}}\neq 0$ if and only if all $\alpha _{n_{j}}\neq 0$, the
comments preceding the statement of the lemma imply that if $\deg \chi _{n}=d\geq 0$, (the case $n=0$ will be left for the reader) we may rewrite the
sum as
\begin{align*}
& \sum_{\overset{(n_{1},\dots,n_{m}) \in \mathbb{N}^{m}}{\sum_{i=1}^{m} \varepsilon _{i}\chi _{n_{i}}=\chi _{n}}} \alpha _{n_{1}}\cdots\alpha _{n_{m}}
\\ & \hspace*{.5in}=\sum_{k=0}^{\infty }\sum_{\overset{\sum_{i=1}^{m}\varepsilon _{i}\chi _{n_{i}}=\chi _{n}}{\max n_{i}\in \lbrack
2^{k},2^{k+1})q_{0}\cdots q_{d-1}\left\vert \chi _{n,d}\right\vert /m}}\alpha _{n_{1}}\cdots\alpha _{n_{m}} \\ & \hspace*{.5in}\leq \sum_{k=0}^{\infty
}\left( 2^{k}q_{0}\cdots q_{d}\left\vert \chi _{n,d}\right\vert /m\right) ^{-s} \left( \sum_{j\leq 2^{k+1}q_{0}\cdots q_{d-1}\left\vert \chi
_{n,d}\right\vert /m}\alpha _{j}\right) ^{m-1} \\ & \hspace*{.5in}\ll \sum_{k=0}^{\infty }\left( 2^{k}q_{0}\cdots q_{d-1}\left\vert \chi
_{n,d}\right\vert /m\right) ^{-s} \left( 2^{k+1}q_{0}\cdots q_{d-1}\left\vert \chi _{n,d}\right\vert /m\right) ^{(1-s)(m-1)} \\ & \hspace*{.5in}\ll
\sum_{k=0}^{\infty }\left( 2^{k}q_{0}\cdots q_{d-1}\left\vert \chi _{n,d}\right\vert /m\right) ^{-\theta m}.
\end{align*}
The final expression is comparable to $n^{-\theta m}$.
\end{proof}

We will first study $r_{N}^{(m)}(\omega )$ with $m=2$.

\begin{proposition}
\label{prop 4} Let $s=\frac{1}{2}+\theta $ where $0<\theta <\frac{1}{2}$. Then, for any $\varepsilon >0$, there exists $K=K(G,s)$ such that
\begin{equation*}
\sum_{N}P(\{\omega \in \Omega :r_{N}^{(2)}(\omega )\geq K\})<\varepsilon .
\end{equation*}
\end{proposition}

\begin{proof}
By definition, we have
\begin{equation}
\begin{split}
r_{N}^{(2)}(\omega )= card & \{(a,b): \chi _{a},\chi _{b}\in E(\omega ), \\
 &  \pm(\chi _{a} +\chi _{b})=\chi _{N},a<b,\text{ or }\chi _{a}-\chi _{b}=\chi
_{N},a\not=b\}.  \label{eqn prop 4-1}
\end{split}
\end{equation}
If $\chi _{a^{\prime }}+\chi _{b^{\prime }}=\chi _{N}$ and $\chi _{a^{\prime }}+\chi _{c^{\prime }}=\chi _{N}$, then $b^{\prime }=c^{\prime }$, and
similarly, if we consider subtraction. Thus, given any $\chi _{a^{\prime }}$ there are at most four ways in which $\chi _{a^{\prime }}$ could appear as
one of $\chi _{a}$ or $\chi _{b}$ in the three equations considered in (\ref {eqn prop 4-1}). Hence, if $r_{N}^{(2)}(\omega )\geq K$, then there exist
at least $\lfloor K/4\rfloor $ pairs $(a_{i},b_{i})$, $1\leq i\leq \lfloor K/4\rfloor $, counted in ~(\ref{eqn prop 4-1}), such that every element of
the set $\{a_{i},b_{j}\}_{1\leq i,j\leq \lfloor K/4\rfloor }$ is distinct. By the pigeon hole principle, one of the three equations considered in (\ref
{eqn prop 4-1}) must be satisfied by at least one third of these $\lfloor K/4\rfloor $ pairs. Without loss of generality, we suppose it is the equation
$\chi _{a}+\chi _{b}=\chi _{N}$ with $a<b$, as the other two cases can be treated in a similar manner. Let $L=\lfloor K/12\rfloor $. By independence,
we have
\begin{align}
P(\{\omega \in \Omega :r_{N}^{(2)}(\omega )\geq K\}) & \leq P(\{\omega \in \Omega :\exists \ L\text{ pairs }(a_{i},b_{i}) \label{eqn prop 4-2'} \\
&\qquad  \text{such that }\chi _{a_{i}},\chi _{b_{i}}\in E(\omega ),  \chi _{a_{i}}+\chi _{b_{i}} =\chi _{N}, \notag \\ & \qquad a_{i}<b_{i},\text{ and
}a_{i},b_{j} \text{ all distinct}\})  \notag \\ &\leq \sum_{S(L)}\prod_{i=1}^{L}P(\{\omega \in \Omega _{m}:\chi _{a_{i}},\chi _{b_{i}}\in E(\omega
)\}),  \notag
\end{align}
where $S(L)$ is the collection of all $L$ distinct pairs, $(a_{i},b_{i})$, $ 1\leq i\leq L$, such that $\chi _{a_{i}}+\chi _{b_{i}}=\chi _{N}$ and $
a_{i}<b_{i}$.

Since the last inequality of (\ref{eqn prop 4-2'}) gives us an $L$-th elementary symmetric function, we can bound it by Lemmas \ref{lemma 4} and
\ref{lemma 3},
\begin{align}
P(\{\omega \in \Omega :r_{N}^{(2)}(\omega )\geq K\}) & \leq \frac{1}{L!} \left( \sum_{S(1)}\alpha _{a}\alpha _{b}\right) ^{L}  \notag \\ &  \notag \\
&\leq
\begin{cases}
\frac{1}{L!}C^{L}, & \mbox{if }N=0 \\ \frac{1}{L!}C^{L}\frac{1}{|N|^{2\theta L}}, & \mbox{if }N\neq 0,
\end{cases}
\notag
\end{align}
for some positive constant $C$. As $\theta >0$, we obtain
\begin{equation}
\sum_{N=-\infty }^{\infty }P(r_{N}^{(2)}(\omega )\geq K)\leq 2\sum_{N=2}^{\infty }\frac{1}{L!}C^{L}\frac{1}{|N|^{2\theta L}}+3\frac{1}{L!}
C^{L}<\varepsilon ,  \label{sum over N}
\end{equation}
for $L$ large enough. Notice that if $G=\bigoplus_{0}^{\infty }\mathbb{Z}(q)$ or $\mathbb{Z}(q^{\infty })$ or $\bigoplus_{0}^{\infty
}\mathbb{Z}(q_{n})$, then we only need to take the sum on the left side of ~(\ref{sum over N}) from $N=0$ to $\infty $.
\end{proof}

\begin{corollary}
\label{corollary 5} Let $s=\frac{1}{2}+\theta $ and $0<\theta <\frac{1}{2}$. Given any $\varepsilon >0$, there exists $K=K(G,s)$ and $\Omega
_{2}\subseteq \Omega $ such that $P(\Omega _{2})\geq 1-\varepsilon $ and $r_{N}^{(2)}(\omega )<K$ for all $N$ and for all $\omega \in \Omega _{2}$.
\end{corollary}

\begin{proof}
By Proposition \ref{prop 4}, we can find $K$ such that $\sum_{N}P(r_{N}^{(2)}(\omega )\geq K)<\varepsilon.$ Let \newline $\Omega _{2}=\{\omega \in \Omega
:r_{N}^{(2)}(\omega )<K$ for all $N\}$. Then, we have
\begin{equation*}
\begin{split}
P(\Omega _{2}^{c}) & =P(\{\omega \in \Omega :\exists N\text{ such that } r_{N}^{(2)}(\omega )\geq K\}) \\ & \leq \sum_{N}P(r_{N}^{(2)}(\omega )\geq
K)<\varepsilon . \qedhere
\end{split}
\end{equation*}
\end{proof}

We complete the induction argument in the following proof. The argument is similar to the base case, but slightly more involved due to the larger value
of $m$. We will then prove the required density property of the subsets in Cor. \ref{theorem 7}.

\begin{theorem}
\label{theorem 6} Let $m\geq 2$ be a positive integer. If $s=\frac{m-1}{m} +\theta $ for $0<\theta <\frac{1}{m}$, then, for all $\varepsilon >0$, there
exists $K_{m}=K_{m}(G,\varepsilon ,s)$ and $\Omega _{m}\subseteq \Omega $ such that $P(\Omega _{m})\geq 1-\varepsilon $ and $r_{N}^{(m)}(\omega )<K_{m}
$ for all $N$ and for all $\omega \in \Omega _{m}$.
\end{theorem}

\begin{proof}
We proceed by induction. Corollary \ref{corollary 5} gives us the base case. Suppose the statement holds for $m_{0}<m$. Fix $\varepsilon >0$. We may
rewrite $s$ as $s=\frac{m_{0}-1}{m_{0}}+\theta ^{\prime }$ with $0<\theta ^{\prime }<\frac{1}{m_{0}}$. Thus, by the inductive hypothesis there exist
$K_{m_{0}}$ and $\Omega _{m_{0}}$ such that $P(\Omega _{m_{0}})\geq 1-\frac{ \varepsilon }{2({m_{0}}+2)}$ and $r_{N}^{({m_{0}})}(\omega )<K_{m_{0}}$
for all $N$ and for all $\omega \in \Omega _{m_{0}}$.

Let $\omega \in \Omega _{m_{0}}$ and fix $t\in \{0,1,\dots,{m_{0}}+1\}$ and integer $N$. Suppose for each $i=1,\dots,K$, we have
\begin{equation}
\chi _{a_{1}^{(i)}}+\cdots+\chi _{a_{t}^{(i)}}-\left( \chi _{a_{t+1}^{(i)}}+\cdots+\chi _{a_{{m_{0}}+1}^{(i)}}\right) =\chi _{N}, \label{thm 6 eqn 2}
\end{equation}
with
\begin{equation}
a_{1}^{(i)}<\cdots<a_{t}^{(i)},a_{t+1}^{(i)}<\cdots<a_{{m_{0}}+1}^{(i)}\text{ and } a_{u}^{(i)}\not=a_{u^{\prime }}^{(i)}\text{ if }u\not=u^{\prime }.
\label{thm 6 eqn 3}
\end{equation}
Assume there exist some $i_{1},\dots,i_{r}$ and $s_{1},\dots,s_{r}$ such that $a_{s_{1}}^{(i_{1})}=a_{s_{j}}^{(i_{j})}$ for all $j\in \{1,\dots,r\}$.
Then, for each $j\in \{1,\dots,r\}$, we have
\begin{equation}
\chi _{a_{1}^{(i_{j})}}+\cdots+\chi _{a_{t}^{(i_{j})}}-\left( \chi _{a_{t+1}^{(i_{j})}}+\cdots+\chi _{a_{{m_{0}}+1}^{(i_{j})}}\right) +\varepsilon
_{j}\chi _{a_{s_{j}}^{(i_{j})}}=\chi _{N}+\varepsilon _{j}\chi _{a_{s_{1}}^{(i_{1})}},  \label{thm 6 eqn 1}
\end{equation}
where $\varepsilon _{j}$ is $-1$ if $s_{j}\leq t$ and $+1$ if $s_{j}>t$, making the left hand side of (\ref{thm 6 eqn 1}) into a combination of
${m_{0}}$ terms. This gives us a total of $r$ representations for $\chi _{N}+\chi _{a_{s_{1}}^{(i_{1})}}$ and $\chi _{N}-\chi _{a_{s_{1}}^{(i_{1})}}$ as a
combination of ${m_{0}}$ terms. By the inductive hypothesis, we have $r\leq 2K_{m_{0}}$. Therefore, it follows that each $({m_{0}}+1)$-tuple,
$(a_{1}^{(i)},\dots,a_{{m_{0}}+1}^{(i)})$, $1\leq i\leq K$, has at most $2({ m_{0}}+1)K_{m_{0}}$ other $({m_{0}}+1)$-tuples, which it may possibly
share an entry with. Hence, by reordering if necessary, there exists a subset of
$L=\Big{\lfloor}\frac{K}{2({m_{0}}+1)K_{m_{0}}}\Big{\rfloor}({m_{0}}+1)$ -tuples, $(a_{1}^{(l)},\dots,a_{{m_{0}}+1}^{(l)})$, $1\leq l\leq L$, which
satisfy (\ref{thm 6 eqn 2}) and ~(\ref{thm 6 eqn 3}), and the additional condition that every element of the set $\{a_{j}^{(l)}\}_{1\leq j\leq
{m_{0}}+1,1\leq l\leq L}$ are distinct.

From the discussion above, and by independence, we have
\begin{align}
P(\{\omega &\in \Omega _{m_{0}}:r_{N,t}^{({m_{0}}+1)}(\omega )\geq K\}) \label{eqn thm 6}
\\
&\leq P ( \{ \omega \in \Omega _{m_{0}}:\exists \ L\
({m_{0}}+1)\text{ -tuples }(a_{1}^{(l)},\dots,a_{{m_{0}}+1}^{(l)}),
\notag \\
& \qquad 1\leq l\leq L,\text{ such that }\sum_{s=1}^{t}\chi
_{a_{s}^{(l)}}-\sum_{s=t+1}^{{m_{0}}+1}\chi _{a_{s}^{(l)}}=\chi _{N},
\notag  \\
& \qquad  \text{ all }a_{j}^{(l)}\text{ are distinct, and }\chi
_{a_{j}^{(l)}}\in E(\omega )\} )
\notag \\
&\leq \sum_{S(L)}\prod_{l=1}^{L}P(\{\omega \in \Omega _{m_{0}}:\chi _{a_{j}^{(l)}}\in E(\omega ),1\leq
j\leq {m_{0}}+1\}),  \notag
\end{align}
where $S(L)$ is the collection of all $L$ distinct $({m_{0}}+1)$-tuples, $(a_{1}^{(l)},\dots,a_{{m_{0}}+1}^{(l)})$, $1\leq l\leq L$, such that
\begin{equation*}
\sum_{s=1}^{t}\chi _{a_{s}^{(l)}}-\sum_{s=t+1}^{{m_{0}}+1}\chi _{a_{s}^{(l)}}=\chi _{N},
\end{equation*}
and $a_{i}^{(l)}\not=a_{j}^{(l)}$ if $i\not=j$.

Since the last inequality of (\ref{eqn thm 6}) gives us an $L$-th elementary symmetric function, we can bound it by Lemmas \ref{lemma 4} and \ref{lemma
3},
\begin{align}
P(\omega \in \Omega _{m_{0}}:r_{N,t}^{({m_{0}}+1)}(\omega )\geq K\}) & \leq \frac{1}{L!}\left( \sum_{S(1)}\alpha _{a_{1}}\dots\alpha _{a_{{m_{0}}
+1}}\right) ^{L}  \notag \\ &\leq
\begin{cases}
\frac{1}{L!}C_{{m_{0}}+1}^{L}, & \mbox{if }N=0, \\ \frac{1}{L!}\left( C_{{m_{0}}+1}\frac{1}{|N|^{({m_{0}}+1)\theta }}\right) ^{L}, & \mbox{if }N\neq
0,
\end{cases}
\notag
\end{align}
for some positive constant $C_{{m_{0}}+1}$.

For each $t\in \{0,1,\dots,{m_{0}}+1\}$, let $\widetilde{\Omega }_{t}=\{\omega \in \Omega _{m_{0}}:r_{N,t}^{({m_{0}}+1)}(\omega )<K$ for all $N\}$. We
can then follow the arguments of Proposition \ref{prop 4} and Corollary \ref {corollary 5} and deduce the existence of $K$ such that for any $t\in
\{0,1,\dots,{m_{0}}+1\}$,
\begin{align}
P(\widetilde{\Omega }_{t}^{c}) &\leq P(\{\omega \in \Omega _{m_{0}}:\exists N\text{ such that }r_{N,t}^{({m_{0}}+1)}(\omega )\geq K\})+P(\Omega /\Omega
_{m_{0}}) \\ &\leq \sum_{N}P(\omega \in \Omega _{m_{0}}:r_{N,t}^{({m_{0}}+1)}(\omega )\geq K\})+\frac{\varepsilon }{2({m_{0}}+2)}  \notag \\
&<\frac{\varepsilon }{{m_{0}}+2}.  \notag
\end{align}
If we let $K_{{m_{0}}+1}=({m_{0}}+2)K$ and $\Omega _{{m_{0}} +1}=\bigcap_{t=0}^{{m_{0}}+1}\widetilde{\Omega }_{t}$, the result follows by (\ref{def
r}).
\end{proof}

\begin{corollary}
\label{corollary 7} Let $m\geq 2$ be a positive integer. If $s=\frac{m-1}{m} +\theta $ for $0<\theta <\frac{1}{m}$, then for a.e.\ $\omega $
\begin{equation*}
\sup_{N}r_{N}^{(m)}(\omega )<\infty .
\end{equation*}
\end{corollary}

\begin{proof}
This follows easily from Theorem \ref{theorem 6}.
\end{proof}

To prove Theorem \ref{thm main intger} we make use of the following variant of the Strong law of large numbers (c.f. \cite[p 140, Thm 11]{HR}).
Notation: $Exp(Y)$ denotes the expectation of the random variable $Y$.
\begin{theorem}
\label{probthm}Let $\{Y_{i}\}$ be simple, independent random variables and $S_{N}=\sum_{i=1}^{N}Y_{i}$. Assume $Exp(Y_{i})>0$, $\lim_{N\rightarrow
\infty }S_{N}=\infty $ and
\begin{equation*}
\sum_{i}\frac{VarY_{i}}{(Exp(S_{i}))^{2}}<\infty.
\end{equation*}
Then $S_{N}/Exp(S_{N})\rightarrow 1$ as $N\rightarrow \infty $ a.e.
\end{theorem}

\begin{corollary}
\label{theorem 7} Let $m\geq 2$ be a positive integer. Given any $\varepsilon >0$, there exists $K_{m}=K_{m}(G,\varepsilon )$ and a set $E(\omega
)\subseteq G^{\prime }$ such that
\begin{equation*}
r_{N}^{(m)}(\omega )<K_{m}
\end{equation*}
for all $N$ and there is a constant $c=c(G,m,\varepsilon )$ such that
\begin{equation}
card\left( \{\chi _{1},\dots,\chi _{n}\}\bigcap E(\omega )\right) \geq cn^{ \frac{1}{m}-\varepsilon }  \label{density}
\end{equation}
for all $n$ when $G=\mathbb{Z}$, $\bigoplus_{0}^{\infty }\mathbb{Z}(q)$ or $\mathbb{Z}(q^{\infty })$, and for infinitely many \thinspace $n$ when
$G=\bigoplus_{0}^{\infty }\mathbb{Z}(q_{j}).$
\end{corollary}

\begin{proof}
Let $s=\frac{m-1}{m}+\varepsilon $. The result follows easily from Theorem \ref{probthm} when $G=\mathbb{Z}$, $\bigoplus_{0}^{\infty }\mathbb{Z}(q)$ or
$\mathbb{Z}(q^{\infty })$.

In the case that $G=\bigoplus_{0}^{\infty }\mathbb{Z}(q_{n})$, let $\{j_{i}\} $ be the indices such that $\alpha _{j_{i}}\neq 0$. Put $Z_{i}=Y_{j_{i}}$
and $S_{N}=\sum_{i=1}^{N}Z_{i}$. Then $VarZ_{i}\leq Exp(Y_{j_{i}})\leq j_{i}^{-s}$ and $Exp(S_{N})=Exp\left( \sum_{i=1}^{j_{N}}Y_{i}\right) $. If we
suppose $j_{N}\in \lbrack q_{0}\cdots q_{d},q_{0}\cdots q_{d+1})$ (which implies $j_{N}\leq (2\left\lfloor \frac{q_{d+1}}{8m}\right\rfloor
+1)q_{0}\cdots q_{d}$), then
\begin{equation}
Exp(S_{N})=\sum_{j<q_{0}\cdots q_{d}}\alpha _{j}+\sum_{j=q_{0}\cdots q_{d}}^{j_{N}}\alpha _{j}.  \label{sum}
\end{equation}
Provided $d$ is suitably large, the first sum is at least
\begin{align*}
\sum_{j=q_{0}\cdots q_{d-1}}^{q_{0}\cdots q_{d}-1}\alpha _{j} &\geq \sum_{j=q_{0}\cdots q_{d-1}}^{q_{0}\cdots q_{d}/(8m) - 1}j^{-s} \\ &\gg \left(
q_{0}\cdots q_{d}/(8m)\right) ^{1-s}-\left( q_{0}\cdots q_{d-1}\right) ^{1-s} \\ &\gg \left( q_{0}\cdots q_{d}\right) ^{1-s}.
\end{align*}
If $q_{0}\cdots q_{d}<j\leq j_{N}$, then we must also have $\alpha _{j}=j^{-s}$. Hence the second sum in (\ref{sum}) is equal to $\sum_{j=q_{0}\cdots
q_{d}+1}^{j_{N}}j^{-s}$ and that is comparable to $j_{N}^{1-s}-\left( q_{0}\cdots q_{d}\right) ^{1-s}$. Putting these together, it follows that
$Exp(S_{N})\gg j_{N}^{1-s}$.

One can also easily check that
\begin{equation*}
\sum_{j\leq q_{0}\cdots q_{d}}\alpha _{j}\ll j_{N}^{1-s},
\end{equation*}
so that $Exp(S_{N})$ is comparable to $j_{N}^{1-s}$. Thus Theorem \ref {probthm} can again be applied to deduce that for a.e. $\omega $,
\begin{equation*}
S_{N}(\omega )=card\left( E(\omega )\bigcap \{\chi _{1},\dots,\chi _{j_{N}}\}\right) \gg j_{N}^{1-s}.
\end{equation*}
In particular, for large $d$,
\begin{align*}
card\left( E(\omega )\bigcap \prod\limits_{n=0}^{d}\mathbb{Z}(q_{n})\right) &=card\left( E(\omega )\bigcap \{\chi _{1},\dots,\chi _{(2\left\lfloor
q_{d}/8m\right\rfloor +1)q_{0}\cdots q_{d-1}}\}\right) \\ &\gg \left( (2\left\lfloor q_{d}/8m\right\rfloor +1)q_{0}\cdots q_{d-1}\right) ^{1-s}\gg
\left( q_{0}\cdots q_{d}\right) ^{1-s}.
\end{align*}
\end{proof}

\section{Application to the existence of thin sets in harmonic analysis\label
{HA}}

In this section, $G$ will denote any discrete abelian group with compact dual group $X$. The groups $\mathbb{Z}$, $\bigoplus_{0}^{\infty }\mathbb{Z}
(q)$, $\bigoplus_{0}^{\infty }\mathbb{Z(}q_{n}\mathbb{)}$, and $\mathbb{Z} (q^{\infty })$ are examples of such discrete groups. The notation
$\widehat{f }$ denotes the Fourier transform of the integrable function $f$ defined on $X $. A subset $E$ of $G$ is said to be a $\Lambda (p)$
\textit{set} for $p>2$ if there is a constant $C_{p}$ such that $\left\Vert f\right\Vert _{p}\leq C_{p}\left\Vert f\right\Vert _{2}$ whenever $f$ is an
$E$ -trigonometric polynomial, meaning $\widehat{f}$ is non-zero only on $E$. As $L^{p}(X)$ $\subseteq L^{q}(X)$ if $p\geq q$, it follows that if $E$
is $\Lambda (p)$, then $E$ is $\Lambda (q)$ for all $q\leq p$.

This notion was introduced for subsets of $\mathbb{Z}$ by Rudin (\cite{Ru}), who proved many important facts about $\Lambda (p)$ sets. In particular,
he showed that if $E\subseteq \mathbb{Z}$ is $\Lambda (p)$, then for all integers $a,d$,
\begin{equation}
card\left( E\bigcap \{a+d,\dots,a+Nd\}\right) \ll N^{2/p}.  \label{Rudin1}
\end{equation}
He also showed that if for some integer $m\geq 2$,
\begin{equation}
\sup_{n\in \mathbb{Z}}\left( card\left\{ (n_{1},\dots,n_{m})\in E^{m}:n=n_{1}+\cdots +n_{m}\right\} \right) <\infty  \label{Rudin}
\end{equation}
then $E$ is $\Lambda (2m)$. Rudin used these properties to construct examples of subsets of $\mathbb{Z}$ which were $\Lambda (2m)$ for a specified
integer $m\geq 2$, but not $\Lambda (2m+\varepsilon )$ for any $\varepsilon >0$. Hajela in \cite{Ha} extended Rudin's properties and constructions to
various other discrete abelian groups, although only achieving the existence of `exact' $\Lambda (2m)$ sets for $m<q$ when $G=\bigoplus_{0}^{\infty
}\mathbb{Z(}q\mathbb{)}$ for $q$ prime, and for $m=2$ when $G=\mathbb{Z}(q^{\infty }).$ Later, Bourgain \cite{Bo} completely settled this problem by
using sophisticated probabilistic methods to prove the existence of exact sets $\Lambda (p)$ for all $p>2$ and for all infinite, discrete abelian
groups $G$.

Using Pisier's operator space complex interpolation, Harcharras in \cite{Har} introduced the notion of completely bounded $\Lambda (p)$ sets: Let
$p>2$. The set $E\subseteq G$ is called \textit{completely bounded} $\Lambda (p)$ (or cb$\Lambda (p)$ for short) if there is a constant $C_{p}$ such
that
\begin{equation*}
\left\Vert f\right\Vert _{L^{p}(X,S_{p})}\leq C_{p}\max \left[ \left\Vert (\sum_{\gamma \in E}\widehat{f}(\gamma )^{\ast }\widehat{f}(\gamma
))^{1/2}\right\Vert _{S_{p}},\left\Vert (\sum_{\gamma \in E}\widehat{f} (\gamma )\widehat{f}(\gamma )^{\ast })^{1/2}\right\Vert _{S_{p}}\right]
\end{equation*}
for all $S_{p}$-valued, $E$-trigonometric polynomials defined on $X$. Here $S_{p}$ denotes the Schatten $p$-class with $\left\Vert T\right\Vert
_{S_{p}}=\left( tr\left\vert T\right\vert ^{p}\right) ^{1/p}$ and
\begin{equation*}
\left\Vert f\right\Vert _{L^{p}(X,S_{p})}=\left( \int_{X}\left\Vert f(x)\right\Vert _{S_{p}}^{p}dx\right) ^{1/p}.
\end{equation*}
Harcharras showed that completely bounded $\Lambda (p)$ sets are always $\Lambda (p),$ but not the converse. She also improved upon Rudin's condition
(\ref{Rudin}) by establishing that $E\subseteq G$ is cb$\Lambda (2m)$ for integer $m\geq 2$ if
\begin{equation}
\sup_{\chi \in G}\left( card\left\{ (\chi _{1},\dots,\chi _{m})\in E^{m}:\chi =\sum_{j=1}^{m}(-1)^{j}\chi _{j}\text{ with }\chi _{j}\text{ distinct}
\right\} \right) <\infty .  \label{Har}
\end{equation}
We remark that this condition was new even for $\Lambda (2m)$ sets. In \cite {Har}, Harcharras used this property to construct examples of subsets of
$\mathbb{Z}$ that were cb$\Lambda (2m)$ but not $\Lambda (2m+\varepsilon )$ for any $\varepsilon >0$. Here we will generalize upon this result
by using (\ref{Har}) to show that every infinite, discrete abelian group $G$ admits a set that is cb$\Lambda (2m)$, but not $\Lambda (2m+\varepsilon )$
for any given $\varepsilon >0$. This will use the work of the previous part of the paper, as well as the following known generalization of
(\ref{Rudin1}).

\begin{lemma}
\label{necessary} \cite{EHR} If $E\subseteq G$ is a $\Lambda (p)$ set for some $p>2$, then there is a constant $C$ such that $card\left( E\bigcap
Y\right) \leq CN^{2/p}$ whenever $Y$ $\subseteq G$ is either an arithmetic progression or a finite subgroup of cardinality $N$.
\end{lemma}

Fix an integer $m\geq 2$ and $\varepsilon >0$. We will first consider $G= \mathbb{Z}$, $\mathbb{Z(}q^{\infty }),$ $\bigoplus_{0}^{\infty }Z(q)$ or
$\bigoplus_{0}^{\infty }Z(q_{n})$, where the $q_{n}$ are strictly increasing, odd primes. We denote the elements of $G^{\prime }$ as $\{\chi
_{n}\}_{n=1}^{\infty },$ as described in the previous section and let $E(\omega )$ be the random sets defined previously, with $s=(m-1)/m+\theta $
where $\theta >0$ is chosen so that $1-s>2/(2m+\varepsilon )$.

\begin{proposition}
For a.e. $\omega $, $E(\omega )$ is cb$\Lambda (2m)$ but not $\Lambda (2m+\varepsilon )$.
\end{proposition}

\begin{proof}
In the notation of (\ref{def r}), Haracharras' condition could be expressed as
\begin{equation*}
E(\omega )\text{ is cb}\Lambda (2m)\text{ if }\sup_{N}r_{N}^{(m)}(\omega )<\infty
\end{equation*}
and we have already seen that $\sup_{N}r_{N}^{(m)}(\omega )$ is finite for a.e. $\omega $ by Cor. \ref{corollary 7}. Also Cor. \ref{theorem 7} shows
that for a.e. $\omega $,
\begin{equation*}
card\left( E(\omega )\bigcap \{\chi _{1},\dots,\chi _{N}\}\right) \gg N^{1-s}
\end{equation*}
when $G=\mathbb{Z}$, $\mathbb{Z(}q^{\infty })$ or $\bigoplus_{0}^{\infty }Z(q)$, and
\begin{equation*}
card\left( E(\omega )\bigcap \{\chi _{1},\dots,\chi _{q_{0}\cdots q_{N}}\}\right) \gg (q_{0}\cdots q_{N})^{1-s}
\end{equation*}
for sufficiently large $N$ when $G=\bigoplus_{0}^{\infty }Z(q_{n})$.

When $G=\mathbb{Z}$, $\{\chi _{1},\dots,\chi _{N}\}$ is an arithmetic progression of length $N$. When $G=\mathbb{Z(}q^{\infty })$ or
$\bigoplus_{0}^{\infty }Z(q)$, $\{\chi _{1},\dots,\chi _{q^{N}-1}\}$ is a subset of a subgroup of cardinality $q^{N}$, and similarly for $G=
\bigoplus_{0}^{\infty }Z(q_{n})$, but with $q^{N}$ replaced by $q_{0}\cdots q_{N}$. In all cases, the choice of $s$ together with Lemma \ref
{necessary} implies that $E(\omega )$ is not $\Lambda (2m+\varepsilon )$ for a.e. $\omega $.
\end{proof}

\begin{theorem}
Let $m\geq 2$ be an integer and $\varepsilon >0$. Every infinite discrete abelian group $G$ contains a set $E$ that is cb$\Lambda (2m)$, but not
$\Lambda (2m+\varepsilon )$.
\end{theorem}

\begin{proof}
As observed in \cite{EHR}, any such group $G$ contains a subgroup isomorphic to one of $\mathbb{Z}$, $\mathbb{Z(}q^{\infty })$, $\bigoplus_{0}^{\infty
}Z(q)$ for $q$ prime, or $\bigoplus_{0}^{\infty }Z(q_{n})$ where $q_{n}$ are strictly increasing, odd primes.

Observe that if $G_{0}$ is a subgroup of $G$ and $f$ is a $S_{p}$-valued $G_{0}$-polynomial, then $f$ is constant on the cosets of $G_{0}^{\bot }$, the
annihilator of $G_{0}$. The same is true for $\left\Vert f\right\Vert _{S_{2m}},$ for any integer $m$. It follows from this that if $E\subseteq G_{0}$
is a cb$\Lambda (2m)$ set, then $E$ viewed as a subset of $G$ is also cb$\Lambda (2m)$ and that $E\subseteq G_{0}$ is a $\Lambda (2m+\varepsilon )$ set
if and only if $E$ viewed as a subset of $G$ is $\Lambda (2m+\varepsilon)$.

Hence it suffices to prove the theorem for the four subgroups listed above, and this was done in the previous proposition.
\end{proof}


\begin{thebibliography}{99}
\bibitem{Bo} J. Bourgain, \textit{Bounded orthogonal systems and the }$ \Lambda (p)$\textit{-set problem}, Acta Math. 162 (1989), 227--245.

\bibitem{ER} P. Erd\H{o}s and A. R\'enyi, \textit{Additive properties of random sequences of positive integers}, Acta Arith. 6 (1960), 83--110.

\bibitem{EHR} R. Edwards, E. Hewitt and K. Ross, \textit{Lacunarity for compact groups I}, Indiana U. Math. J. 21 (1972), 787--806.

\bibitem{Ha} D. Hajela, \textit{Construction techniques for some thin sets in duals of compact abelian groups}, Ann. Inst. Fourier 36 (1986),
    137--166.

\bibitem{Har} A. Harcharras, \textit{Fourier analysis, Schur multipliers on $S^{p}$ and non-commutative $\Lambda (p)$ sets}, Studia Math. 137 (1999),
    203--260.

\bibitem{HR} H. Halberstam and K.F. Roth, \textit{Sequences}, Springer-Verlag, New York, 1983.

\bibitem{PS} C. Pomerance and A. S\'ark\"ozy, \textit{Chapter 20}, in: Handbook of Combinatorics, R. Graham, M. Gr\"otchel and L. Lov\'asz (eds.),
    North Holland, 1995.

\bibitem{Ru} W. Rudin, \textit{Trigonometric series with gaps}, J. Math. and Mech. 9 (1960), 203--227.

\bibitem{Sid} S. Sidon, \textit{Ein Satz \"uber trigonometrische Polynome und seine Anwendung in der Theorie der Fourier-Reihen}, Math. Annalen 106
    (1932), 536--539.

\bibitem{Vu} V.H.Vu, \textit{On a refinement of Waring's problem}, Duke Math. J. 105 (2000), 107--134.

\end{thebibliography}
\end{document}